\documentclass{birkmult}

 \newtheorem{theorem}{Theorem}[section]
 \newtheorem{corollary}[theorem]{Corollary}

 \theoremstyle{definition}
 \newtheorem{definition}[theorem]{Definition}
 \theoremstyle{remark}
 
 \newtheorem{example}{Example}
 \numberwithin{equation}{section}

\usepackage[active]{srcltx}

\usepackage{mathrsfs, amsmath, amssymb}
\usepackage{newsymbol}

\newcommand{\lb}{\langle}
\newcommand{\rb}{\rangle}
\newcommand{\Om}{\Omega}

\newcommand{\F}{\mathscr{F}}
\newcommand{\E}{\mathbb{E}}
\renewcommand{\P}{\mathbb{P}}
\renewcommand{\S}{\Sigma}
\newcommand{\n}{\Vert}

\newcommand{\g}{\gamma}
\newcommand{\s}{^*}
\let\mathcal \undefined
\def\mathcal{\mathscr}

\let\emptyset \undefined
\let\ge       \undefined
\let\le       \undefined
\newsymbol\le          1336  
\newsymbol\ge          133E  
\newsymbol\emptyset    203F
\newsymbol\notle       230A
\newsymbol\notge       230B

\begin{document}

%-------------------------------------------------------------------------
% editorial commands: to be inserted by the editorial office
%
%\firstpage{1}
%\volume{228}
%\DOI{003-0001}
%\seriesextraline{This is the Concrete Title of this Book\br H.E. R and S.T.C. W, Eds.}
%\firstpage{1}
\title[Vector measures of bounded $\g$-variation]{Vector measures of bounded $\g$-variation and stochastic integrals}
%%% add after final submission, for the Arxiv version:
%\footnote{\ \newline 
%This paper will appear in the proceedings of the conference on Vector Measures, Integration and Applications (Eichst\"att 2008).}}

%----------Author 1

\author{Jan van Neerven}
\address{Delft University of Technology\\Delft Institute of Applied Mathematics\\P.O. Box 5031, 2600 GA Delft\\The Netherlands}
\email{J.M.A.M.vanNeerven@TUDelft.nl}

\thanks{The first named author is supported by VICI subsidy 639.033.604 of the Netherlands Organisation for Scientific Research (NWO)}

%----------Author 2

\author{Lutz Weis}
\address{
University of Karlsruhe\\Mathematisches Institut I \\
D-76128 Karlsruhe\\ Germany}
\email{Lutz.Weis@math.uni-karlsruhe.de}



%%% ----------------------------------------------------------------------

\begin{abstract} We introduce the class of vector measures of bounded $\g$-variation and study its relationship with vector-valued stochastic integrals with respect to Brownian motions.
\end{abstract}

%%% ----------------------------------------------------------------------
\maketitle
%%% ----------------------------------------------------------------------

\section{Introduction}

%%%%%% to do: -check [DU]; ref for finite cotype

It is well known that stochastic integrals can be interpreted as vector measures, the identification being given by the identity
$$ F(A) = \int_A \phi\,dB.$$
Here, the driving process $B$ is a (semi)martingale (for instance, a Brownian motion), and $\phi$ is a stochastic process satisfying suitable measurability and integrability conditions. This observation has been used by various authors as the starting point of a theory of stochastic integration  for vector-valued processes.

Let $X$ be a Banach space. In \cite{NeeWei05} we characterized the class of functions $\phi:(0,1)\to X$ which are stochastically integrable with respect to a Brownian motion $(W_t)_{t\in [0,1]}$ as being the class of functions for which the operator $T_\phi: L^2(0,1)\to X$,
$$ T_\phi f := \int_0^1 f(t)\phi(t)\,dt,
$$
belongs to the operator ideal $\g(L^2(0,1),X)$ of all $\g$-radonifying operators. Indeed, we established the It\^o isomorphism
$$ \E \Big\n \int_0^1 f\,dW\Big\n^2 = \n T_f\n_{\g(L^2(0,1),X)}^2.$$
The linear subspace of all operators in $\g(L^2(0,1),X)$ of the form $T=T_f$ for some function $f:(0,1)\to X$ is dense, but unless $X$ has cotype $2$ it is strictly smaller than $\g(L^2(0,1),X)$. This means that in general there are operators $T\in \g(L^2(0,1),X)$ which are not representable by an $X$-valued function.
Since the space of test functions $\mathscr{D}(0,1)$ embeds in $L^2(0,1)$, by restriction one could still think of such operators as $X$-valued distributions. It may be more intuitive, however, to think of $T$ as an $X$-valued vector measure. We shall prove (see Theorem \ref{thm:2} and the subsequent remark) that if $X$ does not contain a closed subspace isomorphic to $c_0$, then the space  $\g(L^2(0,1),X)$ is isometrically isomorphic in a natural way to the space of $X$-valued vector measures on $(0,1)$ which are of bounded $\g$-variation. This gives a `measure theoretic' description of the class of admissible integrands for stochastic integrals with respect to Brownian motions. The condition $c_0\not\subseteq X$ can be removed if we replace the space of $\g$-radonifying operators by the larger space of all $\g$-summing operators (which contains the space of all $\g$-radonifying operators isometrically as a closed subspace).

Vector measures of bounded $\g$-variation behave quite differently from vector measures of bounded variation. For instance, the question whether an $X$-valued vector measure of bounded $\g$-variation can be represented by an $X$-valued function is not linked to the Radon-Nikod\'ym property, but rather to the type $2$ and cotype $2$ properties of $X$ (see Corollaries \ref{type} and \ref{cotype}). 

In section \ref{sec:randomised} we consider yet another class of vector measures whose variation is given by certain random sums, and we show that a function $\phi:(0,1)\to X$ is stochastically integrable with respect to a Brownian motion $(W_t)_{t\in [0,1]}$ on a probability space $(\Omega,\P)$ if and only if the formula $F(A) := \int_A \phi\,dW$
defines an $L^2(\Om;X)$-valued vector measure $F$ in this class.

\section{Vector measures of bounded $\g$-variation}

Let $(S,\S)$ be a measurable space, $X$ a Banach space, and $(\g_n)_{n\ge 1}$ a sequence of independent standard Gaussian random variables defined on a probability space $(\Om,\F,\P)$.

\begin{definition} We say that a countably additive vector measure $F$ has {\em bounded $\g$-variation} with respect to a probability measure $\mu$ on $(S,\S)$ if $\n F\n_{V_\g(\mu;X)}<\infty$, where 
$$ \n F\n_{V_\g(\mu;X)} := \sup \Big(\E \Big\n \sum_{n=1}^N \g_n \,\frac{F(A_n)}{\sqrt{\mu(A_n)}}\Big\n^2\Big)^\frac12,$$
the supremum being taken over all finite collections of disjoint sets $A_1,\dots,A_N\in \S$ such that $\mu(A_n)>0$ for all $n=1,\dots,N$.
\end{definition}

It is routine to check (e.g. by an argument similar to  \cite[Proposition 5.2]{ISEM}) that the space $V_\g(\mu;X)$ of all countably additive vector measures $F:\S\to X$ which have bounded $\g$-variation with respect to $\mu$ is a Banach space with respect to the norm $\n\cdot\n_{V_\g(\mu;X)}$. 
Furthermore, every vector measure which is of bounded $\g$-variation is of bounded $2$-semivariation.

In order to give a necessary and sufficient condition for a vector measure to have bounded $\g$-variation we need to introduce the following terminology.
A bounded operator $T:H\to X$, where $H$ is a Hilbert space, is said to be {\em $\g$-summing} if there exists a constant $C$ such that for all finite orthonormal systems $\{h_1,\dots,h_N\}$ in $H$ one has
$$ \E \Big\n\sum_{n=1}^N \g_n\,Th_n\Big\n^2 \le C^2.$$
The least constant $C$ for which this holds is called the {\em $\g$-summing norm} of $T$, notation $\n T\n_{\g_\infty(H,X)}$.
With respect to this norm, the space $\g_\infty(H,X)$ of all $\g$-summing operators from $H$ to $X$ is a Banach space which contains all finite rank operators from $H$ to $X$. In what follows we shall make free use of the elementary properties of $\g$-summing operators. For a systematic exposition of these we refer to \cite[Chapter 12]{DJT} and the lecture notes \cite{ISEM}. 

\begin{theorem}\label{thm:1}
Let $\mathscr{A}$ be an algebra of subsets of $S$ which generates the $\sigma$-algebra $\S$, and let $F:\mathscr{A}\to X$ be a finitely additive mapping. 
If, for some $1\le p<\infty$, $T:L^p(\mu)\to X$ is a bounded operator such that 
$$F(A) = T1_A, \quad A\in\mathscr{A},$$ then $F$ has a unique extension to a countably additive vector measure on $\S$ which is absolutely continuous with respect to $\mu$. 
 If $T:L^2(\mu)\to X$ is $\g$-summing, then the extension of $F$ has bounded $\g$-variation with respect to $\mu$ and we have
$$ \n F\n_{V_\g(\mu;X)} \le \n T\n_{\g_\infty(L^2(\mu),X)}.$$
\end{theorem}

\begin{proof}
We define the extension $F:\Sigma\to X$ by $F(A):= T1_A$, $A\in\S$.
To see that $F$ is countably additive, consider a disjoint union $A = \bigcup_{n\ge 1}A_n$ with $A_n,A\in\S$. Then
$\lim_{N\to\infty} 1_{\bigcup_{n=1}^N A_n} = 1_A$ in $L^p(\mu)$ and therefore
$$\lim_{N\to\infty} \sum_{n=1}^N F(A_n) = \lim_{N\to\infty} T\sum_{n=1}^N 1_{A_n} = T1_A = F(A).$$
The absolute continuity of $F$ is clear. 
To prove uniqueness, suppose $\tilde F:\S\to X$ is another countably additive vector measure extending $F$. For each $x\s\in X\s$,  $\lb \tilde F,x\s\rb$ and $\lb F,x\s\rb$ are finite measures on $\S$ which agree on $\mathscr{A}$, and therefore by Dynkin's lemma they agree on all of $\S$. This being true for all $x\s\in X\s$, it follows that $\tilde F=F$ by the Hahn-Banach theorem.

Suppose next that $T: L^2(\mu)\to X$ is $\g$-summing, and 
consider a finite collection of disjoint sets $A_1,\dots,A_N$ in $\S$ such that $\mu(A_n)>0$ for all $n=1,\dots,N$.
The functions $f_n = 1_{A_n}/\sqrt{\mu(A_n)}$ are orthonormal in $L^2(\mu)$ and therefore
$$ \E \Big\n \sum_{n=1}^N \g_n \,\frac{F(A_n)}{\sqrt{\mu(A_n)}}\Big\n^2
= \E \Big\n \sum_{n=1}^N \g_n \, Tf_n \Big\n^2
\le \n T\n_{\g_\infty(L^2(\mu),X)}^2.$$
It follows that $F$ has bounded $\g$-variation with respect to $\mu$
and that $\n F\n_{V_\g(\mu;X)}\le  \n T\n_{\g_\infty(L^2(\mu),X)}$. 
\end{proof}

\begin{theorem}\label{thm:2} For a countably additive vector measure $F:\S\to X$ the following assertions are equivalent:
\begin{enumerate}
\item[\rm(1)] $F$ has bounded $\g$-variation with respect to $\mu$;
\item[\rm(2)] There exists a $\g$-summing operator $T:L^2(\mu)\to X$ such that 
$$F(A) = T1_A, \quad A\in\S.$$
\end{enumerate}
In this situation we have $$\n F\n_{V_\g(\mu;X)} = \n T\n_{\g_\infty(L^2(\mu),X)}.$$
\end{theorem}
\begin{proof}
(1)$\Rightarrow$(2): \ Suppose that $F$ has bounded $\g$-variation with respect to $\mu$. For a simple function $f = \sum_{n=1}^N c_n 1_{A_n}$, where the sets $A_n\in\S$ are disjoint and of positive $\mu$-measure,
define
$$ Tf := \sum_{n=1}^N c_n F(A_n).$$
By the Cauchy-Schwarz inequality, for all $x\s\in X\s$ we have
$$
\begin{aligned} 
|\lb T f,x\s\rb| & =  \Big|\E \sum_{m=1}^N \g_m \,c_m\sqrt{\mu(A_m)} \cdot  \sum_{n=1}^N \g_n \,\frac{\lb F(A_n),x\s\rb}{\sqrt{\mu(A_n)}}\Big|
\\ & \le \Big(\E \Big| \sum_{n=1}^N \g_n\, c_n\sqrt{\mu(A_n)} \Big|^2\Big)^\frac12
 \Big(\E \Big| \sum_{n=1}^N \g_n\, \frac{\lb F(A_n),x\s\rb}{\sqrt{\mu(A_n)}} \Big|^2\Big)^\frac12 
\\ &  \le \Big( \sum_{n=1}^N |c_n|^2\mu(A_n)\Big)^\frac12 \n F\n_{V_\g(\mu;X)}\,\n x\s\n
\\ &  = \n f\n_{L^2(\mu)}\n F\n_{V_\g(\mu;X)}\,\n x\s\n.
\end{aligned}
$$
It follows that $T$ is bounded and $\n T\n_{\mathscr{L}(L^2(\mu),X)}\le \n F\n_{V_\g(\mu;X)}$. 
To prove that $T$ is $\g$-summing we shall first make the simplifying assumption that the $\sigma$-algebra $\S$ is countably generated. 
Under this assumption there exists an increasing sequence of finite $\sigma$-algebras $(\Sigma_n)_{n\ge 1}$ such that $\S = \bigvee_{n\ge 1} \S_n$.
Let $P_n$ be the orthogonal projection in $L^2(\mu)$ onto $L^2(\S_n,\mu)$
and put $T_n := T\circ P_n$. These operators are of finite rank and we have $\lim_{n\to\infty} T_n\to T$ in the strong operator topology of $\mathscr{L}(L^2(\mu),X)$.

Fix an index $n\ge 1$ for the moment. Since $\S_n$ is finitely generated there 
exists a partition $S = \bigcup_{j=1}^N A_j$, where the disjoint sets $A_1,\dots,A_N$ generate $\S_n$. 
Assuming that $\mu(A_j)>0$ for all $j=1,\dots,M$ and $\mu(A_j)=0$ for $j=M+1,\dots,N$, the functions $g_j = 1_{A_j}/\sqrt{\mu(A_j)}$, $j=1,\dots,M$, form an orthonormal basis for $L^2(\S_n,\mu)$ and
$$ 
\begin{aligned}
\n T_n\n_{\g_\infty(L^2(\mu),X)}^2 
& = \n T_n\n_{\g_\infty(L^2(\S_n,\mu),X)}^2 \\ 
& = \E \Big\n \sum_{j=1}^M \g_j \, Tg_j \Big\n^2 = \E \Big\n \sum_{j=1}^M \g_j \,\frac{F(A_j)}{\sqrt{\mu(A_n)}}\Big\n^2\le \n F\n_{V_\g(\mu;X)}^2,
\end{aligned}
$$
the first identity being a consequence of \cite[Corollary 5.5]{ISEM}
and the second of \cite[Lemma 5.7]{ISEM}.
It follows that the sequence $(T_n)_{n\ge 1}$ is bounded in $\g_\infty(L^2(\mu),X)$.
By the Fatou lemma, if $\{f_1,\dots,f_k\}$ is any orthonormal family in $L^2(\mu)$, then 
$$ \E \Big\n \sum_{j=1}^k \g_j T f_j\Big\n^2 \le \liminf_{n\to\infty}
\E \Big\n \sum_{j=1}^k \g_j T_n f_j\Big\n^2 \le \n T_n\n_{\g_\infty(L^2(\mu),X)}^2
\le  \n F\n_{V_\g(\mu;X)}^2.
$$
This proves that $T$ is $\g$-summing and
$\n T\n_{\g_\infty(L^2(\mu),X)}\le \n F\n_{V_\g(\mu;X)}$.

It remains to remove the assumption that $\S$ is countably generated.
The preceding argument shows that if we define $T$ in the above way, then 
its restriction to $L^2(\S',\mu)$ is $\g$-summing for every countably generated $\sigma$-algebra $\S'\subseteq\S$, with a uniform bound
$$\n T\n_{\g_\infty(L^2(\S',\mu),X)}\le \n F\n_{V_\g(\mu;X)}.$$
Since every finite orthonormal family $\{f_1,\dots,f_k\}$ in 
$L^2(\mu)$ is contained in $L^2(\S',\mu)$ for some countably generated $\sigma$-algebra $\S'\subseteq\S$, we see that
$$ 
\E \Big\n \sum_{j=1}^k \g_j\, T f_j\Big\n^2 
\le \n T\n_{\g_\infty(L^2(\S',\mu),X)}^2 \le\n F\n_{V_\g(\mu;X)}^2.
$$
It follows that $T$ is $\g$-summing and
$\n T\n_{\g_\infty(L^2(\mu),X)}\le \n F\n_{V_\g(\mu;X)}$.

(2)$\Rightarrow$(1): \ This implication is contained in Theorem \ref{thm:1}.
\end{proof}

By a theorem of Hoffmann-J{\o}rgensen and Kwapie\'n  \cite[Theorem 9.29]{LedTal}, if $X$ is a Banach space not containing an isomorphic copy of $c_0$, then for any Hilbert space $H$ one has 
$$ \g_\infty(H,X) = \g(H,X),$$
where by definition $\g(H,X)$ denotes the closure in $\g_\infty(H,X)$ of the finite rank operators from $H$ to $X$. Since any operator in this closure is compact we obtain:

\begin{corollary}
If 
$X$ does not contain an isomorphic copy of $c_0$ and $F:\S\to X$ has bounded $\g$-variation with respect to $\mu$, then $F$ has relatively compact range.
\end{corollary}

Using the terminology of \cite{NeeWei05}, a theorem of Rosi\'nski and Suchanecki  \cite{RosSuc} asserts that if $X$ has type $2$ we have
a continuous inclusion $L^2(\mu;X)\hookrightarrow \g(L^2(\mu),X)$ and that if $X$ has cotype $2$ we have a continuous inclusion $\g_\infty(L^2(\mu),X)\hookrightarrow L^2(\mu;X)$. In both cases the embedding is contractive, and the relation between the operator $T$ and the representing function $\phi$ is given by
$$ Tf = \int_S f\phi\,d\mu, \quad f\in L^2(\mu).$$
If $\dim L^2(\mu)=\infty$, then in the converse direction the existence of a continuous embedding
$L^2(\mu;X)\hookrightarrow \g_\infty(L^2(\mu),X)$ (respectively $\g(L^2(\mu),X)\hookrightarrow L^2(\mu;X)$) actually implies the type $2$ property (respectively the cotype $2$ property) of $X$. 

\begin{corollary}\label{type} Let $X$ have type $2$. For all $\phi\in L^2(\mu;X)$ the formula
$$ F(A):= \int_A \phi\,d\mu, 
\quad A\in\S,$$ defines a countably additive vector measure $F:\S\to X$ which has bounded $\g$-variation with respect to $\mu$. Moreover,
$$ \n F\n_{V_\g(\mu;X)}\le \n \phi\n_{L^2(\mu;X)}.$$
If $\dim L^2(\mu)=\infty$, this property characterises the type $2$ property of $X$.
\end{corollary}
\begin{proof} By the theorem of Rosi\'nski and Suchanecki, $\phi$ represents an operator $T\in \g(L^2(\mu),X)$ such that
$ T1_A = \int_A \phi\,d\mu = F(A)$ for all $ A\in\S.$ 
The result now follows from Theorem \ref{thm:1}.
The converse direction follows from Theorem \ref{thm:2} and the preceding remarks.
\end{proof}

\begin{corollary}\label{cotype} Let $X$ have cotype $2$. If $F:\S\to X$ has bounded $\g$-variation with respect to $\mu$, there exists a function $\phi\in L^2(\mu;X)$
such that $$F(A) = \int_A \phi\,d\mu, \quad A\in\S.$$
 Moreover,
$$  \n \phi\n_{L^2(\mu;X)}\le \n F\n_{V_\g(\mu;X)}.$$
If $\dim L^2(\mu)=\infty$, this property characterises the cotype $2$ property of $X$.
\end{corollary}
\begin{proof}
By Theorem \ref{thm:2} there exists an operator $T\in \g_\infty(L^2(\mu),X)$
such that $F(A) = T1_A$ for all $A\in\S$. Since $X$ has cotype $2$, $X$ does not contain an isomorphic copy of $c_0$ and therefore the theorem of Hoffmann-J{\o}rgensen and Kwapie\'n implies that $T\in \g(L^2(\mu),X)$.
Now the theorem of Rosi\'nski and Suchanecki shows that $T$ is represented by
a function $\phi\in L^2(\mu;X)$. The converse direction follows from Theorem \ref{thm:1} and the remarks preceding Corollary \ref{type}.
\end{proof}

\section{Vector measures of bounded randomised variation}\label{sec:randomised}

Let $(S,\S)$ be a measurable space and $(r_n)_{n\ge 1}$ a Rademacher sequence, i.e., a sequence of independent random variables with $\P(r_n = \pm 1) = \frac12$.

\begin{definition}
A countably additive vector measure $F:\S\to X$ is of {\em bounded randomised variation} if $ \n F\n_{V^{\rm r}(\mu;X)}<\infty$, where
$$ \n F\n_{V^{\rm r}(\mu;X)} = \sup\Big(\E \Big\n \sum_{n=1}^N r_n\, F(A_n)\Big\n^2\Big)^\frac12,$$
the supremum being taken over all finite collections of disjoint sets $A_1,\dots,A_N\in\S$.
\end{definition}

Clearly, if $F$ is of bounded variation, then $F$ is of bounded randomised variation. The converse fails; see Example \ref{ex:1}. If $X$ has finite cotype, standard comparison results for Banach space-valued random sums \cite{DJT, LedTal} imply that an equivalent norm is obtained when the Rademacher variables are replaced by Gaussian variables.

It is routine to check that the space $V^{\rm r}(\mu;X)$ of all countably additive vector measures $F:\S\to X$ of bounded randomised variation is a Banach space with respect to the norm $\n\cdot\n_{V^{\rm r}(\mu;X)}$.

In Theorem \ref{thm:3} below we establish a connection between measures of 
bounded randomised variation and the theory of stochastic integration. For this purpose we need the following terminology.
A {\em Brownian motion} on $(\Om,\F,\P)$ indexed by another probability space $(S,\S,\mu)$ is a mapping
$W: \S\to L^2(\Om)$ such that:
\begin{enumerate}
\item[\rm(i)] For all $A\in \S$ the random variable $W(A)$ is centred Gaussian
with variance $$\E (W(A))^2 = \mu(A);$$  
\item[\rm(ii)] For all disjoint $A,B\in\S$ the random variables $W(A)$ and $W(B)$ are independent.
\end{enumerate} 
A strongly $\mu$-measurable function $\phi:S\to X$ is {\em stochastically integrable}
with respect to $W$ if for all $x\s\in X\s$ we have $\lb \phi,x\s\rb\in L^2(\mu)$ (i.e, $f$ {\em belongs to $L^2(\mu)$ scalarly}) and for all $A\in\S$ there exists a strongly measurable random variable $Y_A:\Om\to X$ such that
for all $x\s\in X\s$ we have
$$ \lb Y_A,x\s\rb = \int_A \lb \phi,x\s\rb\,dW$$ almost surely.
Note that each $Y_A$ is centred Gaussian and therefore belongs to $L^2(\Om;X)$ by Fernique's theorem; the above equality then holds in the sense of $L^2(\Om)$. We define the {\em stochastic integral} of $\phi$ over $A$ by $\int_A\phi\,dW := Y_A.$ For more details and various equivalent definitions we refer to \cite{NeeWei05}.

\begin{theorem}\label{thm:3} Let $W:\Sigma\to L^2(\Om)$ be a Brownian motion. For a strongly $\mu$-measurable function $\phi:S\to X$ the following assertions are equivalent:
\begin{enumerate}
\item[\rm(1)] $\phi$ is stochastically integrable with respect to $W$;
\item[\rm(2)] $\phi$ belongs to $L^2(\mu)$ scalarly and there exists a countably additive vector measure $F:\S\to X$, of bounded $\g$-variation with respect to $\mu$,
such that for all $x\s\in X\s$ we have 
$$ \lb F(A),x\s\rb = \int_A \lb\phi,x\s\rb\,d\mu, \quad A\in \S;$$
\item[\rm(3)] $\phi$ belongs to $L^2(\mu)$ scalarly and there exists a countably additive vector measure $G:\S\to L^2(\Om;X)$ of bounded randomised variation such that for all $x\s\in X\s$ we have
$$ \lb G(A),x\s\rb = \int_A \lb \phi,x\s\rb\,dW, \quad A\in \S.$$
\end{enumerate}
In this situation we have
$$ \n F\n_{V_\g(\mu;X)} = \n G\n_{V^{\rm r}(\mu;L^2(\Om;X))} = \Big(\E\Big\n \int_S \phi\,dW\Big\n^2\Big)^\frac12.$$
\end{theorem}
\begin{proof}
(1)$\Leftrightarrow$(2): \ This equivalence is immediate from Theorem \ref{thm:2} and the fact, proven in \cite{NeeWei05}, that $\phi$ is stochastically integrable
with respect to $W$ if and only there exists an operator $T\in \g(L^2(\mu),X)$ such that 
$$ Tf = \int_S f\phi\,d\mu, \quad f\in L^2(\mu).$$  
In this case we also have $$\n T\n_{\g(L^2(\mu),X)} = \Big(\E\Big\n \int_S \phi\,dW\Big\n^2\Big)^\frac12.$$ In view of Theorem \ref{thm:2}, this proves the identity $$\n F\n_{V_\g(\mu;X)} = \Big(\E\Big\n \int_S \phi\,dW\Big\n^2\Big)^\frac12.$$

(1)$\Rightarrow$(3): \ Define $G:\S\to L^2(\Om;X)$ by
$$ G(A) := \int_A  \phi\,dW, \quad A\in \S.$$
By the $\g$-dominated convergence theorem \cite{NeeWei05}, $G$ is countably additive.
To prove that $G$ is of bounded randomised variation we consider
disjoint sets $A_1,\dots,A_N\in\S$.
If  $(\tilde r_n)_{n\ge 1}$ is a Rademacher sequence on a probability space $(\tilde\Om,\tilde\F,\tilde \P)$, then by randomisation we have
$$
\begin{aligned}
\tilde\E \Big\n \sum_{n=1}^N \tilde r_n\, G(A_n)\Big\n_{L^2(\Om;X)}^2
& = \tilde\E \E \Big\n \sum_{n=1}^N \tilde r_n \int_{A_n} \phi\,dW\Big\n^2
\\ & =  \E \Big\n \sum_{n=1}^N \int_{A_n} \phi\,dW\Big\n^2
\le \E \Big\n \int_{S} \phi\,dW\Big\n^2,
\end{aligned}
$$
with equality if $\bigcup_{n=1}^N A_n = S$.
In the second identity we used that the $X$-valued random variables $\int_{A_n} \phi\,dW$ 
are independent and symmetric. The final inequality follows by, e.g., covariance domination \cite{NeeWei05} or an application of the contraction principle. 
It follows that $G$ is a countably additive vector measure 
of bounded randomised variation and
$$ \n G\n_{V^{\rm r}(\mu;X)} = \Big(\E\Big\n \int_S \phi\,dW\Big\n^2\Big)^\frac12.$$
   
(3)$\Rightarrow$(1): \ This is immediate from the definition of stochastic integrability. 
\end{proof}

\begin{example}\label{ex:1} If $W$ is a standard Brownian motion on $(\Om,\F,\P)$ indexed by the Borel interval $([0,1],\mathscr{B},m)$, then $W$ is a countably additive vector measure with values in $L^2(\Om)$ which is of bounded randomised variation, but of unbounded variation. The first claim follows from Theorem \ref{thm:3} since $W(A) = \int_A 1\,dW$ for all Borel sets $A$. To see that $W$ is of unbounded variation, note that for any partition $0= t_0<t_1<\dots< t_{N-1}<t_N= 1$
we have
$$ \sum_{n=1}^N \n W((t_{n-1},t_n))\n_{L^2(\Om)}= 
\sum_{n=1}^N \sqrt{t_n-t_{n-1}}.$$
The supremum over all possible partitions of $[0,1]$ is unbounded.
\end{example}


\begin{thebibliography}{CoifRoch}
 
\bibitem{DieUhl} J. Diestel and J.J. Uhl, 
\textit{Vector measures}.
Mathematical Surveys, Vol. 15,  Amer. Math. Soc., Providence (1977).

\bibitem{DJT} J. Diestel, H. Jarchow and A. Tonge,
\textit{Absolutely summing operators}. 
Cambridge Studies in Adv. Math., Vol. 34, Cambridge, 1995.

\bibitem{LedTal} M. Ledoux and M. Talagrand, 
\textit{Probability in Banach spaces.}
Ergebnisse d. Math. u. ihre Grenzgebiete, Vol. 23, Springer-Verlag, 1991. 

\bibitem{ISEM} J.M.A.M. van Neerven,
\textit{Stochastic evolution equations.} Lecture notes of the 11th International Internet Seminar, 
TU Delft, downloadable at {\tt http://fa.its.tudelft.nl/\~{\,}\!isemwiki}.

\bibitem{NeeWei05} J.M.A.M. van Neerven and L. Weis,
\textit{Stochastic integration of functions with values in a Banach space.} 
Studia Math. {\bf 166} (2005), 131--170.
 
\bibitem{RosSuc} J. Rosi\'nski and Z. Suchanecki, 
\textit{On the space of
vector-valued functions integrable with respect to the white noise.}
Colloq. Math. {\bf 43} (1980), 183--201.
\end{thebibliography}
\end{document}